\documentclass[12pt]{article}

\usepackage{latexsym,amssymb,amsmath}

\newcommand{\C}{\mathbb C}
\newcommand{\R}{\mathbb R}
\newcommand{\Z}{\mathbb Z}
\newcommand{\Q}{\mathbb Q}
\newcommand{\N}{\mathbb N}
\newcommand{\F}{\mathbb F}

\newcommand{\slt}{SL(2,\mathbb C)}
\newcommand{\sma}{\left(\begin{array}}
\newcommand{\fma}{\end{array}\right)}

\newtheorem{lem}{Lemma}[section]

\newtheorem{co}[lem]{Corollary}
\newtheorem{thm}[lem]{Theorem}
\newtheorem{prop}[lem]{Proposition}

\newenvironment{proof}{\textbf{Proof.}}{\newline\hspace*{\fill}{$\Box$}\\}

\begin{document}
\title{A 3-manifold group which is not four dimensional linear}
\author{J.\,O.\,Button\\
Selwyn College\\
University of Cambridge\\
Cambridge CB3 9DQ\\
U.K.\\
\texttt{jb128@dpmms.cam.ac.uk}}
\date{}
\maketitle
\begin{abstract}
We give examples of closed orientable graph 3-manifolds with fundamental
group which is not a subgroup of $GL(4,k)$ for any field $k$. This
answers a question in the Kirby problem list from 1977
which is credited to the late William Thurston. 
\end{abstract}
\section{Introduction}

As part of Thurston's revolutionary understanding of 3 dimensional
geometry and topology, he established that
the fundamental groups of compact
3-manifolds could exhibit much better behaviour than 
typically found amongst all finitely presented groups. In particular two
properties of increasing strength that a finitely presented (or generated)
group can hold are being residually finite and being linear (which here
will always mean a subgroup of $GL(n,\F)$ for $\F$ any field). To give
an idea of his influence, it is hard to imagine now that before this work
the residual finiteness of $\pi_1(M)$ for $M$ the exterior of a prime
knot would only have been established for torus knots and the very few knots 
known to be hyperbolic, such as the those in the computational
work of R.\,Riley. The fundamental group of
a hyperbolic 3-manifold must be a subgroup of $P\slt$ that lifts to
$\slt$ and so is linear in 2 dimensions and residually finite, so
Thurston's results showing that many 3-manifolds have hyperbolic structures
raised the possibility that the fundamental group of every compact 3-manifold
could be residually finite or even linear. Indeed in the first version of the
Kirby problem list which dates from 1977, we have Question 3.33 by Thurston
with Part (A) asking whether all such groups can be embedded in $GL(4,\R)$
and Part (B) asking whether these groups are all residually finite. This
latter question was established to be true on the acceptance of Perelman's
solution to Geometrisation, because Thurston indicated and Hempel proved in
\cite{hemrf} that residual finiteness is preserved when constructing
3-manifolds from their geometric pieces which themselves will have linear
fundamental group.

Now this does not show the stronger property of linearity which is still
open for 3-manifold groups. However a large amount of recent activity
means that there is only a very restricted range of possible counterexamples.
Moreover this recent progress, using work \cite{wis} of Wise on virtually
special groups, unexpectedly shows linearity over $\Z$. Yi Liu proved
in \cite{liu} that graph manifolds with a metric of non positive curvature
have fundamental group linear over $\Z$. With Agol's work in \cite{agar}
establishing this for hyperbolic 3-manifolds and \cite{ppw} covering
3-manifolds with at least one hyperbolic piece of the JSJ decomposition,
this only leaves closed 3-manifolds without virtually special fundamental
group, which by these results is equivalent to not possessing a metric
of non positive curvature. Although some Siefert fibre spaces fall into
this category, their fundamental groups are known to be linear over $\Z$
so the question of linearity of finitely generated 3-manifold groups
is now only open for closed graph manifolds which do not have a metric of non
positive curvature. Thus it might be that every 3-manifold group is linear
over $\Z$ which would have been a big surprise until very recently.
However a word of warning may be in order because these faithful
$\Z$-representations are obtained through a long process of argument and
are expected to be of extremely high dimension, so would not be
easy to construct directly.

Let us now move from linearity over $\Z$ to $\C$, where we have the same
3-manifold groups for which linearity is still open. Assuming that they
are linear or removing them from the discussion as appropriate, we can 
further ask: is there $n\in\N$ such that all finitely generated 3-manifold
groups embed in $GL(n,\C)$? This question does seem to take on a
different flavour because if $H\leq GL(n,\C)$ is an index $i$ subgroup of $G$
then we can say straight away that $G$ is linear because the induced
representation shows that $G\leq GL(in,\C)$. However this is no good for
our new question unless we have a bound on $i$. Now all hyperbolic 3-manifold
groups embed in $\slt$ but other 3-manifold groups can too: for instance
in \cite{busl} we showed that on identifying two copies of the figure 8
knot along the common torus boundary, the resulting closed 3-manifold has
a fundamental group which sometimes does and sometimes does not embed in
$\slt$, depending on the identification. As $\slt\leq GL(4,\R)$ because
a $\C$-linear map of $\C^2$ is an $\R$-linear map of $\R^4$, it seems
reasonable to pose as Thurston did in the Kirby problem list the question
of whether every finitely generated 3-manifold group embeds in $GL(4,\R)$.
(We know of no reference to this in Thurston's own writing, nor is there
any indication of whether he thought it true or false.)

In this paper we show that this question has a negative answer, even on
replacing $\R$ with any field of any characteristic. We describe the
3-manifold in Section 2 and give details of its fundamental group but
here we can summarise it thus: take two copies of the product of the
punctured torus and the circle and form the graph manifold by identifying
the boundary tori, with some conditions on the monodromy. The whole
argument relies only on using Jordan normal form up to 4 by 4 matrices
and considering the centraliser of a matrix in Jordan normal form.
However the key idea is this: the circle in the first product 3-manifold
requires a matrix having a large centraliser (by which we mean it
contains a non abelian free group). But as we do not allow ourselves to
identify the two circles, this centraliser cannot contain the whole 
3-manifold group. This argument applies also to the circle on the other
side and if these two elements are both diagonalisable then they are
simultaneously diagonalisable as they commute. This forces a block
structure for all the matrices in each of the two pieces of the graph
manifold and in Section 3 we show by an easy examination of the possible
cases for both block decompositions in 4 dimensions that this cannot
occur, because the diagonal entries of the circle elements will be roots
of unity so these elements will have finite order.

In Section 4 we show how this conclusion that the eigenvalues are roots of
unity generalises to arbitrary matrices over an algebraically closed
field, by replacing simultaneous diagonalisation by simultaneous
triangularisation. This then allows the positive characteristic case to
be eliminated first in Section 5 by a quick argument, leaving only the
field $\C$ without loss of generality. However in Section 6 we now have to
deal with the circle elements having more complicated Jordan normal forms.
Although their eigenvalues are still roots of unity, such matrices may
of course have infinite order in the characteristic zero case. This
section requires some rather more specialised arguments which we feel
would not extend quickly to dimensions above 4, unlike those in the
earlier sections. However the basis of these arguments is just taking each
possible Jordan normal form for the circle elements and working out the
centralisers.

In the last section we make a few related comments, including noting that
our graph manifolds have already appeared in the literature where they were
shown not to have a metric of non positive curvature (so the linearity of
these 3-manifold groups is still open) and to be non fibred but virtually
fibred.

\section{Description of the graph manifolds}

We can form a closed orientable graph 3-manifold in the following way: let
$S_{g,1}$ be the compact orientable surface of genus $g\geq 1$ with one
boundary component. We know that $\pi_1(S_{g,1})=F_{2g}$, the free group
of rank $2g$, and we let $A\in F_{2g}$ be the element given by the
boundary curve (oriented in some way). On forming the product manifold
$M_1=S_{g,1}\times S_1$ (which can be regarded as a trivial Siefert
fibre space) we have that the group $G_1=\pi_1(M_1)$ is isomorphic to
$F_{2g}\times\Z$ with the element $S$ generating $\Z$ being in the
centre of this fundamental group. Moreover we have 
$\langle A,S\rangle=\Z\times\Z$ as this forms the fundamental group of the
boundary torus $\partial M_1$.

We now take another manifold $M_2$ of this type with fundamental group
$G_2$ (here the genus $g'$ of our new surface $S_{g',1}$ does not need
to equal $g$, although in Section 6 we will require that $g=g'=0$)
with $B$ the corresponding 
peripheral element of $S_{g',1}$ and $T$ the equivalent generator of the
centre of $\pi_1(S_{g',1}\times S_1)$. Let $M$ be the closed 
orientable graph
manifold $M_1\#_fM_2$ where $f:\partial M_1\rightarrow\partial M_2$ is an
orientation reversing homeomorphism of the torus which identifies the
boundaries of the two 3-manifolds. This means that $\pi_1(M)$ is equal to
the amalgamated free product $(F_{2g}\times\Z)*_\theta(F_{2g'}\times\Z)$
where $\theta:\langle A,S\rangle\rightarrow\langle B,T\rangle$ is an
isomorphism. The automorphisms of $\Z\times\Z$ are of course elements
of $GL(2,\Z)$ so we have integers $i,j,k,l$ with $il-jk=\pm 1$
such that $B=A^iS^j$ and
$T=A^kS^l$. Although we might need to worry about the sign as regards
the orientability of $M$, we can assume that $il-jk=1$ for $\pi_1(M)$
because we can replace $T$ by $T^{-1}$ (thus $k$ and $l$ by $-k$ and $-l$)
without changing the group. For here on we do not consider 3-manifolds as
we only need to examine the group $G=\pi_1(M)$, although we now must note
the fact that $A$, and $B$, is equal to a product of commutators
in $G_1$, respectively $G_2$.
 
\section{Diagonalisable peripheral elements over $\C$}

One of the most basic but useful facts in linear algebra is that if
two $n\times n$ matrices $X,Y$
with entries in a field $\F$ are separately
diagonalisable over $\F$ and they commute then they are simultaneously
diagonalisable. This is because $Y$ maps the $\lambda$-eigenspace
$E_\lambda(X)=\{v\in\F^n:Xv=\lambda v\}$ of $X$ into itself and so we can 
diagonalise $Y$ when restricted to each $E_\lambda(X)$. Thus on taking our
group $G$, we assume for the remainder of this section
that both the peripheral elements $A$ and $S$ of $\pi_1(M_1)$
are diagonalisable over $\F$, with the other cases being
dealt with in the later sections. We take $\F=\C$ for definiteness here
and as the case of most interest, but the arguments are valid in any field
on interpreting the phrase root of unity as an element of finite order in
the multiplicative group $\F-\{0\}$. We further assume 
throughout the rest of the paper
that none of $i,j,k,l\in\Z$ are equal to zero.

Consequently if $G$ embeds as a subgroup of $GL(n,\C)$ for some $n$,
we can conjugate $G$ so that both $A$ and $S$ are diagonal matrices, as well as
$B$ and $T$ which are products of these two elements. Now all of the non
abelian free group $F_{2g}$ commutes with $S$, so that $S$ cannot have 
distinct diagonal entries which would imply an abelian centraliser in
$GL(n,\C)$. Thus we can choose a basis $e_1,\ldots ,e_n$ of eigenvectors of 
$S$ with basis elements picked from each of the $d$ eigenspaces $U_i$ 
and this provides a direct sum decomposition of
$\C^n$ as $U_1\oplus\cdots\oplus U_d$. Now as any $g\in G_1$ commutes
with $S$ we have that $g$ splits into $d$ square block matrices with
respect to this basis. However the same argument also applies to $T$ and
$G_2$ with the same basis, except we would need to apply a permutation
of $\{1,\ldots ,n\}$ if we wanted to group the $d'$ eigenspaces 
$V_1\oplus\cdots\oplus V_{d'}$ of $T$
together, according to the repeated entries on the diagonal of $T$.

The idea now is to look at the equations satisfied by the diagonal 
entries of the four matrices $A,B,S,T$. We have $2n+d+d'$ variables 
in these entries with the identities $B=A^iS^j$, $T=A^kS^l$ providing 
$2n$ equations, which are linear homogeneous equations for elements in 
the abelian group $\C-\{0\}$, written multiplicatively. Therefore it seems
we ought to be able to find non trivial
solutions to these equations, but we also need to recall that $A$ and $B$
are elements in the commutator subgroup of $G_1$ and $G_2$ respectively.
As the determinant is a homomorphism from $GL(n,\C)$ to $\C-\{0\}$, we
must have $\mbox{det}(A)=\mbox{det}(B)=1$. Moreover as all elements of
$G_1$ have the same block structure, each block of $A$ is also a
product of commutators 
in $GL(m,\C)$ for the relevant $m<n$ and so they all have
determinant 1 as well. The same applies to $B$ so we now have $d+d'$
further homogeneous equations to satisfy, thus we no longer can
guarantee the existence of non trivial solutions. 

We have seen that $d,d'<n$ and we also note that $d$ (and $d'$) is greater
than 1 as otherwise
$S$ (or $T$) would be a scalar multiple $\lambda I_n$ of the identity.
But $B=A^iS^j$ with $\mbox{det}(A)$ and $\mbox{det}(B)$ equal to 1
implying that $\mbox{det}(S)$ is a root of unity. Thus $\lambda$ is a
root of unity too which means that $S$ has finite order. 
Similarly the decomposition of $\C^n$ into the $T$-eigenspaces $\{V_j\}$
cannot be the same as that for the $S$-eigenspaces $\{U_i\}$, nor can we
have every $V_j$ contained in some $U_i$ for $i$ depending on $j$.
Otherwise every element of $G_1$ and $G_2$ preserves these $S$-eigenspaces,
thus the restrictions of both $A$ and $B$ to each $U_i$ have determinant 1
and the argument just given applies to show that the eigenspace $U_i$
corresponds to an eigenvalue which is a root of unity. As this applies for
all $U_i$, once again $S$ has finite order.
Of course the same holds with $S$ and $T$ swapped.

As a warm up we first
consider the lower dimensions. For two dimensions, note that
$G_1=F_{2g}\times\Z$ does embed in $GL(2,\C)$ but only by taking $S$ equal
to $\lambda I$ (for $\lambda$ not a root of unity) so this would imply that
$S$ commutes with all of $G$.

\begin{prop}
The group $G$ as above does not embed in $GL(3,\C)$ if the elements $A,S$
are both diagonalisable over $\C$ and $j\neq 0$.
\end{prop}
\begin{proof}
We do the simultaneous diagonalisation as above, along with the splitting
into blocks, with respect to the basis $e_1,e_2,e_3$. Without loss of
generality we can only have elements of $G_1$ having block structure
$\langle e_1,e_2\rangle\oplus\langle e_3\rangle$ and 
$\langle e_1,e_3\rangle\oplus\langle e_2\rangle$ for $G_2$, because we saw
above that the block structures cannot be the same.
On imposing the
determinant condition for each block of $A$ and $B$, our four peripheral
elements must have the following form, where we are writing diagonal
matrices as column vectors:
\[A=\sma{r} a\\1/a\\1 \fma,S= \sma{c}\lambda\\ \lambda \\ \mu\fma;
B=\sma{r} \alpha \\1\\1/\alpha\\ \fma, T= \sma{c} \eta\\ \theta\\ \eta \fma\]
for $a,\lambda,\mu,\alpha,\eta,\theta\in\C-\{0\}$. Now the identity $B=A^iS^j$
implies that $(\mbox{det}\,S)^j=1$. For the moment we ignore complex roots
of unity other than one and thus assume that $\mu=1/\lambda^2$. Then the
double appearance of $\eta$ in $T$ implies that $a^k\lambda^{3l}=1$ but
the 1 in $B$ means that $a^i=\lambda^j$, thus $\lambda^{kj+3li}=1$. As
$kj+3li=4li-1$, this means that $\lambda$ (and from here all other
variables) is a root of unity and hence
$S$ is the identity under our assumption unless $li=1/4$ which cannot
happen.

To deal with roots of unity in general, suppose that we have found a solution
to the entries of $A,S,B,T$ in which roots of unity appear. Let $N$ be the 
least common multiple over all the orders of these roots. Then on replacing
$A,S,B,T$ with $A^N,S^N,B^N,T^N$ the same equations between these matrices
continue to hold, as well as the determinant being 1 in each block of
$A^N$ and $B^N$, but now all roots of unity have been replaced by 1.
\end{proof} 

\begin{thm} The group $G$ above does not embed in $GL(4,\C)$ if the
elements $A, S$ are both diagonalisable over $\C$ and none of
$i,j,k,l$ are equal to 0.
\end{thm}
\begin{proof}
After the simultaneous diagonalisation we have both the block structure
for $G_1$ and for $G_2$. The possibilities for the block sizes on either side,
which need not be the same, are (3,1), (2,2) and (2,1,1). We first note that
if the same block of size 1 appears in both $G_1$ and $G_2$ then there is an
entry of 1 in the same place of $A$ and of $B$ (coming from the determinant
condition), thus $A^iS^j=B$ means that the equivalent entry of $S$ is also
(a root of) unity, and thus so is that of $T$ using $A^kS^l=T$. Thus on
taking this block of size 1 and the complementary block of size $n-1$
(which here is 3), we obtain an expression of $G$ as a subdirect product
inside $H_1\times H_2$, where $H_1$ is the image of $G$ under the homomorphism
which is restriction to the size 1 block, and similarly $H_2$ for the
size $n-1$ block. But $H_1$ being abelian implies that
the latter homomorphism is an isomorphism,
because an element of $G$ that restricts to $I_{n-1}$ will be in the
centre of $G$. Consequently we 
we can delete the size 1 block from our matrices without changing the
group, thus reducing the problem to one lower
dimension which is covered by Proposition 3.1.

We now outline the argument in each of the possible cases and will henceforth
write all equations additively.\\
Case 1: Both blocks are (2,2).\\
Let the decomposition for $S$ be $\langle e_1,e_2\rangle\oplus\langle
e_3,e_4\rangle$. As noted earlier, this cannot be the same for $T$ so
without loss of generality we set this to be 
$\langle e_1,e_4\rangle\oplus \langle e_2,e_3\rangle$. This gives rise to
diagonal matrices
\[A=\sma{r} a\\-a\\b\\-b \fma,S= \sma{c}\lambda\\ \lambda \\ \mu\\ \mu\fma;
B=\sma{r} \alpha \\ \beta\\-\beta\\ -\alpha\fma, 
T= \sma{c} \eta\\ \theta\\ \theta\\ \eta \fma.\]
As before, considering determinants of $S$ and $T$, assuming $j\neq 0$ and 
ignoring roots of unity gives $\mu=-\lambda$ and $\theta=-\eta$. Then the
appearance of $\pm\alpha$ at the top and bottom of $B$ implies
$ia+j\lambda=\alpha=ib+j\lambda$, thus $a=b$ as $i\neq 0$. But now the
two $\eta$s in $T$ give $ka+i\lambda=\eta=-ka-i\lambda$ so $\eta=0=\theta$,
meaning that $T$ is the identity (or a power of $T$ is
on removing roots of unity). Now the top two equations for the entries
of $T$ tell us that $a$ and $\lambda$ are also roots of unity, so everything
is.\\
Case 2: Both blocks are (3,1).\\
We need not consider them to be the same decomposition so we can set them
equal to $\langle e_1,e_2,e_3\rangle\oplus\langle e_4\rangle$
for $S$ and $\langle e_1,e_2,e_4\rangle\oplus \langle
e_3\rangle$ for $T$, giving entries
\[A=\sma{r} a\\b\\-a-b\\0 \fma,S= \sma{c}\lambda\\ \lambda \\ \lambda\\ 
-3\lambda\fma;
B=\sma{r} \alpha \\ \beta\\0\\ -\alpha-\beta\fma, 
T= \sma{c} \eta\\ \eta\\ -3\eta\\ \eta \fma\] 
where we can again assume all determinants are 1, as $j\neq 0$. Also
$k\neq 0$ gives $a=b$ by examining the top two entries of $T$ so 
$\alpha=\beta$. Thus $ka+l\lambda=\eta=-3l\lambda$ and $-2ia+j\lambda=0$,
giving $(8il+jk)\lambda=0$. As $9il\neq 1$ we have $S=I$.\\
Case 3: Both blocks are (2,1,1).\\
As we do not need to put blocks of size 1 together, we will take
$\langle e_1,e_4\rangle\oplus\langle e_2\rangle\oplus\langle e_3\rangle$
for $S$ and $\langle e_2,e_3\rangle\oplus\langle e_1\rangle
\oplus\langle e_4\rangle$ for $T$ and set
\[A=\sma{r} a\\0\\0\\-a \fma,S= \sma{c}\lambda\\ \mu\\ \nu\\ 
\lambda\fma;
B=\sma{r} 0\\ \alpha \\ -\alpha\\0\fma, 
T= \sma{c} \theta\\ \eta\\ \eta\\ \nu \fma\]
which gives $ia+j\lambda=0=-ia+j\lambda$, forcing $A=I$ and all entries
are again roots of unity.\\
Case 4: Blocks of form (2,2) and (2,1,1).\\
We can swap $G_1$ and $G_2$ if necessary so that $S$ has the (2,2) blocks.
This has the effect of replacing the matrix $\sma{rr}i&j\\k&l\fma$ with its
inverse, so the entries will still be non zero.
Then we have without loss of generality
\[A=\sma{r} a\\-a\\b\\-b \fma,S= \sma{c}\lambda\\ \lambda\\ -\lambda\\ 
-\lambda\fma;
B=\sma{r} \alpha \\ 0\\0\\-\alpha\fma, 
T= \sma{c} \theta\\ \eta\\ \nu\\ \theta \fma\]
so the middle two entries of $B$ give $a=b$ and then the outer entries of $T$
imply $\theta=0$. But now we have $ia=j\lambda$ and $ka+l\lambda=0$ which
means that $(2il-1)a=0$ so again $A=I$ and we only have roots of unity.\\
Case 5: Blocks of form (3,1) and (2,1,1).\\
Here we can assume we have blocks $\langle e_1,e_2,e_3\rangle\oplus
\langle e_4\rangle$ for $S$ and $\langle e_1,e_4\rangle\oplus\langle
e_2\rangle\oplus\langle e_3\rangle$ for $T$ and set
\[A=\sma{r} a\\b\\-a-b\\0 \fma,S= \sma{c}\lambda\\ \lambda\\ \lambda\\ 
-3\lambda\fma;
B=\sma{r} \alpha \\ 0\\0\\-\alpha\fma, 
T= \sma{c} \theta\\ \eta\\ \nu\\ \theta \fma\]
which implies using $\alpha$ that $ia=2j\lambda$ and $ka+4l\lambda=0$ using
$\theta$, so $(2kj+4il)\lambda=0$ but $6il\neq 2$ so $S=I$ and only roots
of unity appear.\\
Case 6: Blocks of form (3,1) and (2,2).\\
On setting blocks of $\langle e_1,e_2,e_3\rangle\oplus\langle e_4\rangle$
and $\langle e_1,e_2\rangle\oplus\langle e_3,e_4\rangle$ we obtain
\[A=\sma{r} a\\b\\-a-b\\0 \fma,S= \sma{c}\lambda\\ \lambda\\ \lambda\\ 
-3\lambda\fma;
B=\sma{r} \alpha \\ -\alpha\\ \beta\\-\beta\fma, 
T= \sma{c} \theta\\ \theta\\ \eta\\ \eta \fma\]
so the repeated $\theta$ gives $a=b$ and the repeated $\beta$ gives 
$ia+j\lambda=0$, but the repeated $\eta$ implies $ak=2l\lambda$ so
$(3il-1)\lambda=0$, giving $S=I$ with all roots of unity again.\\
\end{proof}

\section{Diagonal entries in the non diagonal case}
  
We assumed throughout the last section that our peripheral elements
could be diagonalised but now we will see that in the general case
we can still use standard linear algebra to conclude that
all eigenvalues of these peripheral elements are again roots of unity.

Given any element $X\in GL(n,\C)$, or if given an arbitrary field $\F$
we can replace $\C$ by the relevant algebraic closure $\overline{\F}$,
we know as part of the theory of Jordan
normal form that $\C^n$ is spanned by its generalised eigenspaces
\[G_{\lambda}(X)=\{v\in\C^n:(X-\lambda I)^mv=0\mbox{ for some } m\in\N\}.\]
Now on taking another $Y\in GL(n,\C)$ such that $XY=YX$, we have
$Y(G_\lambda(X))\subseteq G_\lambda(X)$ because $Y$ also commutes with
$(X-\lambda I)^m$. Thus if $V_1\oplus\cdots\oplus V_e$ is the decomposition
of the peripheral element 
$S$ into its generalised eigenspaces then this also
provides a block decomposition for all elements in the group $G_1$. Thus
not only does this apply to $A\in G_1$ but it also means that each square
block in $A$ has determinant 1 as before. Unlike the diagonalisable case
though, it is not guaranteed that an element which decomposes into these
blocks commutes with $S$.

Now we consider $T=A^kS^l$ which also commutes with $S$, thus preserves the
decomposition $V_1\oplus\cdots\oplus V_e$. Hence on restricting 
$S$ and $T$ to
each $V_c$, this subspace further splits into a decomposition
$V_c=W_{c,1}\oplus\cdots\oplus W_{c,n_c}$ where $W_{c,m}$ is contained in a
single generalised eigenspace of $T$. As $W_{c,m}$ is the intersection of
a generalised eigenspace for $S$ and that for $T$, it is invariant under
$S,T$ and any element commuting with both of these.
We also note that as any element $g\in G_2$ commutes with $T$, the
decomposition $V'_1\oplus\cdots\oplus V'_{e'}$ into generalised eigenspaces
for $T$ also provides a sum into square blocks for $g$. Thus as before
each block of $B$ is a commutator and so has determinant 1.

We now restrict $A$ and $S$ to each subspace $W_{c,m}$ and this is
invariant under both maps. Although our matrices need not be diagonalisable,
we know that over $\C$ (or $\overline{\F}$)
any matrix can be conjugated to be upper triangular.
We would like to have
an equivalent version of the ``commute and diagonalisable implies
simultaneously diagonalisable'' result in the upper triangular case and
this is provided by reference to Wikipedia 
(pagename ``Simultaneous triangularisability'' which gives the credit to
Frobenius). One can indeed
replace ``diagonalisable'' with ``upper triangularisable'' and
``simultaneously diagonalisable'' with 
``simultaneously upper triangularisable'' in the above for two
commuting matrices $X,Y$ because we have a non trivial eigenspace
$E_\lambda(X)$ containing the first element of a basis making $X$ upper
triangular. This eigenspace is invariant under $Y$, so on restricting $Y$
we obtain a common eigenvector for $X$ and $Y$. By
taking this as our first basis vector, we can now take a quotient space
and reduce the dimension by one, then continue by induction.

We now apply this result to $A$ and $S$ restricted to each $W_{c,m}$ and
put these bases together
together to obtain a basis for $\C^n$ where these two matrices are
upper triangular. This means that $B=A^iS^j$ and $T=A^kS^l$ are also
upper triangular with the equations for the diagonal elements of $B$
and $T$ involving only the diagonal elements of $A$ and $S$, so these
equations will be exactly the same as in the last section. 
Moreover the determinant
equations for $A$ and $B$ will involve restricting to a basis for each $V_c$,
respectively $V'_c$, and we can do this by taking the appropriate elements
of the basis from each $W_{c,m}$. Thus we are taking determinants of upper
triangular subblocks which is just the product of diagonal terms, so any set
of variables and equations for the diagonal elements has already appeared
in Section 3. This gives us:
\begin{thm} If our group $G$ embeds in $GL(n,\F)$ for any field $\F$ with
algebraic closure $\overline{\F}$ then it can be conjugated in
$GL(n,\overline{\F})$ so that the peripheral elements $A,S$ (and hence
$B,T$) are upper triangular with all diagonal elements roots of unity in
$\overline{\F}$.
\end{thm}
\begin{proof}
As just mentioned, we can assume by conjugating $G$ that $A$ and $S$ are
both upper triangular. The arguments given in Section 3 that the number
of eigenspaces of $S$ (and $T$) is strictly between 1 and $n$ still applies
here, as does the argument that not every eigenspace of $T$ is contained
in an $S$-eigenspace (or vice versa) because any determinant which needs
to be evaluated will just be a product of elements on the diagonal of
$A,S,B$ or $T$. Thus the argument given in Theorem 3.2 that the only
solutions to the equations for the diagonal elements yield roots of
unity applies here too.
\end{proof} 

\section{The positive characteristic case}

We are now in a position to first eliminate $G$ being a subgroup of
$GL(4,\F)$ when $\F$ has positive characteristic. This follows from
the following lemma for which we are unable to provide a reference:
presumably it is so well known to those working in positive characteristic
that it might not need writing down.
\begin{lem} Suppose that $X$ is a matrix in $GL(n,\F)$ where $\F$ has
positive characteristic $p$.
If all the eigenvalues of $X$ are roots of unity
in $\overline{\F}$ then $X$ has finite order.
\end{lem}
\begin{proof}
We conjugate $X$ in $GL(n,\overline{\F})$ so that it is upper triangular,
thus of the form $X=D+N$ where $D$ is diagonal and $N$ is upper triangular
with all diagonal entries equal to zero. By taking an appropriate power of
$X$, we can assume that $D=I$ and we know that $N^m=0$ for all $m\geq n$.
But $X^{p^k}=(I+N)^{p^k}=I+N^{p^k}$ because ${p^k} \choose r$ is 0 modulo
$p$ for all $0<r<p^k$ so we just need to take $p^k\geq n$.
\end{proof}
\begin{co} Our group $G$ does not embed in $GL(4,\F)$ if $\F$ is any field
of positive characteristic.
\end{co}
\begin{proof}
On applying Theorem 4.1, Lemma 5.1 immediately tells us that $A,S,B,T$
have finite order.
\end{proof} 

\section{The characteristic zero case}

If $G$ embeds in $GL(4,\F)$ for $\F$ a field of characteristic zero then
we can say without loss of generality that $\F=\C$. This is because
$G$ is finitely generated so we can embed 
$\Q(x_1,\ldots ,x_n)$ into $\C$, where $x_1,\ldots ,x_n$ are the matrix
entries of a generating set for $G$. Unfortunately we are not yet finished
because of matrices such as $\sma{cc}1&1\\0&1\fma$ which have infinite
order despite all eigenvalues being roots of unity. We will need to
develop some rather ad hoc arguments in this section in order to finish
our proof (as well as introducing parity constraints on $i,j,k,l$ to
avoid having to treat too many cases). However it will be of great help
to mention now some standard but very useful facts about 2 by 2 matrices
which we will apply. At this point we set the genera $g,g'$ of the two
surfaces in Section 2 equal to 1 so that the peripheral elements $A$ and
$B$ are now commutators and not just products of commutators. Now suppose
$\gamma\in SL(2,\C)$ is a commutator $\alpha\beta\alpha^{-1}\beta^{-1}$
for $\alpha,\beta\in GL(2,\C)$. First $\mbox{det}(\gamma)=1$ but now suppose
$\gamma$ has repeated eigenvalues which must be both 1 or both $-1$. In the
first case $\langle\alpha,\beta\rangle$ must be a soluble group; indeed if
we conjugate so that $\gamma=\sma{cc}1&b\\0&1\fma$ then both $\alpha$ and
$\beta$ are forced to be upper triangular too. In the second case we have
the following lemma from the author's PhD (the appendix in \cite{bph}).

\begin{lem}
If $\alpha,\beta\in SL(2,\C)$ and 
\[
\alpha\beta\alpha^{-1}\beta^{-1}=\left( \begin{array}{rr}
-1 & -b \\ 0 & -1 \end{array} \right)\]
for some $b\in\C$
then we can simultaneously conjugate $\alpha$ and $\beta$ in $SL(2,\C)$
such that $\alpha\beta\alpha^{-1}\beta^{-1}$ is as above and
\[ \alpha^{-1}\beta^{-1}\alpha\beta=\left( \begin{array}{rr}
-1 & 0 \\ -b & -1\fma\]
(for a possibly different $b\in\C$)
whereupon there exist $z,w\in \C-\{0\}$ such that
\[\alpha= \left( \begin{array}{cc}
\frac{1+z^2}{w} & z \\
z & w
\end{array} \right),
\quad
\beta= \left( \begin{array}{cc}
\frac{1+w^2}{z} & -w \\
-w & z
\end{array} \right)
\mbox{ with }b=\frac{2(1+z^2+w^2)}{zw}.
\]
\end{lem}

We also here note the obvious but useful fact that if we have a homomorphism
from a non abelian free group $F$ to $GL(n,\C)$ with soluble kernel then
this is injective because all subgroups of $F$ are free and $F$ has no
normal cyclic subgroups apart from the identity.

We first sort the possible Jordan blocks (generalised eigenspaces with a
suitable basis putting them into a canonical form) that can appear in a
4 by 4 matrix into 2 categories: those with small 
(read does not contain a non abelian free group)
centraliser and those with big (contains a non abelian free group) centraliser.

Those with small centraliser are
$
\sma{c}\lambda\fma,\sma{cr}\lambda&1\\0&\lambda\fma,
\sma{ccc}\lambda&1&0\\0&\lambda&0\\0&0&\lambda\fma,$
\[\sma{ccc}\lambda&1&0\\0&\lambda&1\\0&0&\lambda\fma, 
\sma{cccc}\lambda&1&0&0\\0&\lambda&1&0\\0&0&\lambda&0\\
0&0&0&\lambda\fma,
\sma{cccc}\lambda&1&0&0\\0&\lambda&1&0\\0&0&\lambda&1\\
0&0&0&\lambda\fma,
\]
and those with big centraliser are
\[\lambda I_2,\,\lambda I_3,\,\lambda I_4,\,
\sma{cccc}\lambda&1&0&0\\0&\lambda&0&0\\0&0&\lambda&0\\
0&0&0&\lambda\fma,
\sma{cccc}\lambda&1&0&0\\0&\lambda&0&0\\0&0&\lambda&1\\
0&0&0&\lambda\fma.
\]
Notice that of the 4 by 4 blocks we can distinguish between small and
big centraliser according to whether the matrix does not or does
satisfy the polynomial $(t-\lambda)^2$ respectively. For the latter
matrices we will need to know the exact description of each centraliser.
This is obvious for the first three but in the last two cases we will
change the canonical form to obtain a neater description (it has always
struck the author as slightly curious that given the two conjugate matrices
\[\sma{ccc}\lambda&1&0\\0&\lambda&0\\0&0&\lambda\fma
\mbox{ and }
\sma{ccc}\lambda&0&1\\0&\lambda&0\\0&0&\lambda\fma\]
the centraliser of the first does not consist solely of upper
triangular matrices whereas it does for the second).

We will conjugate the penultimate matrix into the form
\[\sma{cccc}\lambda&0&0&1\\0&\lambda&0&0\\0&0&\lambda&0\\
0&0&0&\lambda\fma,\mbox{ which has centraliser }
\sma{cccc}a&?&?&?\\0&?&?&?\\0&?&?&?\\0&0&0&a\fma\\
\]
where ? denotes any complex number, not necessarily the same number on
each appearance, whereas repeated letters are equal to each other.
Meanwhile the final matrix will instead be written
\[\sma{cccc}\lambda&0&1&0\\0&\lambda&0&1\\0&0&\lambda&0\\
0&0&0&\lambda\fma,\mbox{ which has centraliser }
\sma{cccc}a&b&?&?\\c&d&?&?\\0&0&a&b\\0&0&c&d\fma.\\
\]
We can now prove our main result.
\begin{thm} Let $G$ be the amalgamated free product
\[(F(X,Y)\times F(S))*_{H_1=H_2}(F(U,V)\times F(T))\]
where $F(X_1,\ldots , X_n)$ denotes the free group on elements $X_1,\ldots
,X_n$ and $H_1=\langle A=XYX^{-1}Y^{-1},S\rangle\cong\Z\times\Z\cong
H_2=\langle B=UVU^{-1}V^{-1},T\rangle$ 
with the identification of $H_1$ to $H_2$
given by $B=A^iS^j$ and $T=A^kS^l$ for $il-jk=1$ with $j,k$ odd integers
and $i,l$ non zero even integers. Then $G$ is the fundamental group
of a closed orientable graph 3-manifold but is not a subgroup of
$GL(4,\F)$ for $\F$ any field.
\end{thm} 
\begin{proof}
By the previous results we can assume that $\F=\C$ and that all eigenvalues
of $A,S,B,T$ are roots of unity. In particular none of these matrices is
diagonalisable else it would be of finite order. However we can conjugate
$G$ in $GL(4,\C)$ such that $S$ is in Jordan normal form (although using
our modified Jordan blocks rather than the standard ones where they differ).
Thus $S$ is a direct sum of Jordan blocks as given above, but as the
centraliser of $S$ contains a non abelian free group, we need to pick out
at least one block from the lower list. We start with the cases where
$S$ has only one eigenvalue.\\
Case 1:
\[S=
\sma{cccc}\lambda&0&0&1\\0&\lambda&0&0\\0&0&\lambda&0\\
0&0&0&\lambda\fma
\mbox{ and }F(X,Y)\mbox{ is of the form }
\sma{cccc}a&?&?&?\\0&?&?&?\\0&?&?&?\\0&0&0&a\fma\\
\]
which means that there is a homomorphism from $F(X,Y)\leq GL(4,\C)$
to $GL(2,\C)$ given by restriction to the middle 2 by 2 square of these
matrices. As any element in the kernel would be upper triangular, this
homomorphism is injective. Moreover if we denote the image of $w\in F(X,Y)$
by $\overline{w}$, we have that $\overline{A}$ is a commutator
$[\overline{X},\overline{Y}]$ of 2 by 2 matrices $\overline{X},\overline{Y}$
freely generating a free group. In particular $\mbox{det}(\overline{A})=1$
but we do not have 1 as a repeated eigenvalue by our comment on 2 by 2
matrices earlier. However the eigenvalues $\rho,\rho^{-1}$ of
$\overline{A}$ are also eigenvalues for $A$ so must be roots of unity by
Theorem 4.1. But $\overline{A}$ has infinite order so we cannot have
distinct eigenvalues, leaving only that $\overline{A}$ is conjugate to
$\sma{rr}-1&-b\\0&-1\fma$ for $b\neq 0$. Moreover the top and bottom
diagonal entries of $A$ are 1 because $A=[X,Y]$ for $X,Y$ in the centraliser
of $S$.

We now do some tidying up before proceeding to examine the forms of $B,T$
and the centraliser of $T$. First we conjugate by elements of the form
\[
\sma{cccc}1&\sigma&0&0\\0&1&0&0\\0&0&1&\tau\\0&0&0&1\fma\\
\mbox{ to ensure that }
A=\sma{cccc}1&0&?&?\\0&?&?&?\\0&?&?&0\\0&0&0&1\fma.\\
\]
We then conjugate using a 2 by 2 matrix in the middle block (and the
identity outside this), so that $A$ is of the same form but with its
middle block equal to $\overline{A}$. Note that we have stayed within 
the centraliser of $S$ and left it unchanged throughout,  thus
\[
B=\sma{cccc}\lambda^j&0&?&?\\0&\lambda^j&ib\lambda^j&?\\
0&0&\lambda^j&0\\0&0&0&\lambda^j\fma\mbox{ and }
T=\sma{cccc}\lambda^l&0&?&?\\0&-\lambda^l&-kb\lambda^l&?\\
0&0&-\lambda^l&0\\0&0&0&\lambda^l\fma\]
because $i$ is even and $k$ is odd. We now produce enough detail of the
centraliser of $T$ to complete the argument in this case. If we work with
2 by 2 blocks in order to save excessive variable names, we are looking to
see when
\[\sma{c|c} A&B\\ \hline C&D\\\fma
\mbox{ and }T=\sma{c|r}J&E\\ \hline 0&-J\fma
\mbox{ commute, where }
J=\sma{rr}\lambda^l&0\\0&-\lambda^l\fma\neq 0.\]
Now looking at the bottom left hand corner of the two equal products tells
us that $CJ=-JC$ so $C$ must have zeros on the diagonal. However from the
top left of the products we have $AJ-JA=EC$ and because the bottom left
entry of $E$ is non zero, this forces the top right hand entry of $C$ to
equal zero. Putting this back into $AJ-JA=EC$ means that the top right
hand entry of $A$ is also zero and the same argument provides the same
conclusion for the top right hand entry of $D$. Thus elements
commuting with $T$, and in particular $U$ and $V$, are all of the form
\[\sma{cccc}?&0&?&?\\?&?&?&?\\0&0&?&0\\?&0&?&?\fma.\\
\]
We note two points here: invariance of the second basis vector and that
there is a homomorphism from the centraliser of $T$ to the entries in the
four corners. Thus $B=[U,V]$ for $U,V$ of this form forces the entry in the
second column of $B$ to be 1, thus $\lambda^j=1$. But the kernel of our
homomorphism is soluble, because commutators of commutators
of elements in the kernel are
all upper triangular. Thus the homomorphism is injective when restricted
to $F(U,V)$ but the four corners of $B$ form the matrix 
$\sma{cc}1&?\\0&1\fma$, which cannot be a commutator of a pair of elements
in $GL(2,\C)$ generating a non abelian free group.\\
Case 2:
\[S=
\sma{cccc}\lambda&0&1&0\\0&\lambda&0&1\\0&0&\lambda&0\\
0&0&0&\lambda\fma
\mbox{ and }F(X,Y)\mbox{ is of the form }
\sma{cccc}a&b&?&?\\c&d&?&?\\0&0&a&b\\0&0&c&d\fma.\]
Thus we have an obvious homomorphism from $F(X,Y)$ to $GL(2,\C)$ by
restriction to the top left top 2 by 2 block (which is equal to the
bottom right block). Once again we denote the image of $w\in F(X,Y)$ by
$\overline{w}$ and note that $A=[X,Y]$ implies that
\[A=
\sma{c|c} [\overline{X},\overline{Y}]=\overline{A}
&?\\ \hline 0&[\overline{X},\overline{Y}]\fma
\]
with $\overline{A}$ having determinant 1. On conjugating $\overline{A}$
into Jordan normal form and building a block matrix in the centraliser
of $S$ with this 2 by 2 conjugating element 
repeated twice on the diagonal and zeros elsewhere, we can assume that
\[A=
\sma{rrrr}-1&-b&a_{11}&a_{12}\\0&-1&a_{21}&a_{22}\\0&0&-1&-b\\
0&0&0&-1\fma
\mbox{ and }T=
\sma{cccc}-\lambda^l&-kb\lambda^l&t_{11}&t_{12}\\0&-\lambda^l&t_{21}&t_{22}\\
0&0&-\lambda^l&-kb\lambda^l\\0&0&0&-\lambda^l\fma\] 
as $k$ is odd. Now $T$ would need to have a big centraliser so we are done
unless $(T+\lambda^lI)^2=0$, which would mean either $kb\lambda^l=0$
(which is false) or $t_{21}=0$ and $t_{11}+t_{22}=0$.
At this point we felt it reasonable to revert
to the computer. First suppose that $\overline{X}$ and 
$\overline{Y}$ are both in $SL(2,\C)$. Using a similar block conjugating
matrix as above, we can assume by Lemma 6.1 that
\[X=
\sma{cccc}\frac{1+z^2}{w}&z&x_{11}&x_{12}\\z&w&x_{21}&x_{22}\\
0&0&\frac{1+z^2}{w}&z\\0&0&z&w\fma
\mbox{ and }Y=
\sma{cccc}\frac{1+w^2}{z}&-w&y_{11}&y_{12}\\-w&z&y_{21}&y_{22}\\
0&0&\frac{1+w^2}{z}&-w\\0&0&-w&z\fma.\]
On feeding this into Mathematica and asking for the (simplified version of)
$A=XYX^{-1}Y^{-1}$, we found that
\begin{eqnarray*}
a_{11}+a_{22}=\frac{2(1+z^2+w^2)}{zw}
(&-&x_{12}w-x_{21}w+2y_{22}w+2x_{22}z+y_{12}z+y_{21}z)\\
\mbox{and\qquad}a_{21}=&-&x_{12}w-x_{21}w+2y_{22}w+2x_{22}z+y_{12}z+y_{21}z
\end{eqnarray*}
so $a_{21}=0$ implies that $a_{11}+a_{22}=0$. Now it is easily shown by
induction on $k$ and $l$ that $a_{21}\neq 0$ implies $t_{21}\neq 0$ in
which case we are fine. But again arguing by induction on $k$ and $l$, if
$a_{21}=a_{11}+a_{22}=0$ then the same equations hold for the equivalent
entries in all powers of $A$. Thus here $t_{11}+t_{22}=-2l\lambda^{l-1}$ which
is non zero as $l\neq 0$. Thus the centraliser of $T$ is never big enough
in this case.

We did assume above that $\overline{X},\overline{Y}\in SL(2,\C)$
but for block matrices of the form
$\sma{c|c} A&B\\ \hline 0&A\fma$
the determinant is $(\mbox{det}\,X)^2$.
Thus on being given $X,Y\in GL(4,\C)$
which both lie in the centraliser of $S$, 
we can multiply each by an appropriate scalar such that the repeated
diagonal 2 by 2 block in each matrix has determinant 1 and the
resulting commutator $XYX^{-1}Y^{-1}$ will be unchanged.\\
Case 3: The final case is where $S$ has more than one eigenvalue, hence we
are putting Jordan blocks together. But we require at least one block from
the lower list and at least one non diagonal block, which only allows for two
eigenvalues $\lambda\neq\mu$ with
\[S=
\sma{cccc}\lambda&0&0&0\\0&\lambda&0&0\\0&0&\mu&1\\
0&0&0&\mu\fma
\mbox{ and }F(X,Y)\mbox{ of the form }
\sma{cccc}?&?&0&0\\?&?&0&0\\0&0&a&?\\0&0&0&a\fma.\]
Again $\lambda,\mu$ are roots of unity and 
$A$ not of finite order and $\langle
X,Y\rangle$ not soluble implies that
we can conjugate to get
\[A=
\sma{cccc}-1&-b&0&0\\0&-1&0&0\\0&0&1&0\\
0&0&0&1\fma\mbox{ and }T=
\sma{cccc}-\lambda^l&?&0&0\\0&-\lambda^l&0&0\\0&0&\mu^l&?\\
0&0&0&\mu^l\fma.\]
But both ? are non zero so if $-\lambda^l\neq\mu^l$ then the
centraliser of $T$ is too small. If however they are equal then $T$ can
now be conjugated to have the same canonical form as $S$ did in Case 2.
Therefore we may as well swap $S$ and $T$, which can be achieved by
reversing the order of the two factors in the amalgamated product, thus
keeping $G$ to be the same group but replacing the gluing matrix by its
inverse. However the entries will still all be non zero and the same
conditions on their parities will continue to hold.
\end{proof}
\section{Further comments}

It seems strange that our proof has relied solely on the basic arguments
of Jordan normal form and centralisers. However the main advance here is
identifying a 3-manifold whose fundamental group is susceptible to this
approach. We briefly say how this came about, because the original aim
was to prove linearity of graph 3-manifolds.

The question of whether the fundamental group of every closed 3-manifold
is linear remains open. However this is known in nearly all cases:
linearity is preserved by subgroups, finite index supergroups and free
products so we quickly find ourselves only having to worry about
closed orientable irreducible 3-manifolds $M$
with infinite fundamental group.
We can then invoke Geometrisation which tells us that $M$ is either
a Siefert fibred space, is hyperbolic or admits a JSJ decomposition
where each piece has one of these two structures. Now $G=\pi_1M$ is
linear in the Siefert fibred case (see \cite{afw} Theorem 8.7
for a neat proof, credited to Boyer, that $G$ is linear over $\Z$)
and of course embeds in $\slt$ if $M$ is hyperbolic. In fact the recent
work of Wise, Agol and others also gives us linearity over $\Z$.
Shortly after this a paper \cite{ppw} of Przytycki and Wise obtained
linearity (again over $\Z$) for $G$ when the JSJ decomposition contains
at least one piece which is hyperbolic. Thus this only leaves the case
where all pieces are Siefert fibred, namely graph manifolds.

Consequently this paper started as an attempt to show linearity 
(at least over $\C$) of the fundamental group of a graph manifold $M$. A
useful start is that one can take a finite cover of $M$ where all the
pieces in the JSJ decomposition are $S^1$ bundles over a surface. As these
will all have boundary, we can reduce to the case where each piece is
a product and so we can assume that $G$ is a graph of groups
with vertex groups of the form $F_n\times\Z$ and $\Z\times\Z$ edge groups.
We looked at the case where the graph is a tree, so that $G$ is formed by
repeated amalgamations over these $\Z\times\Z$ subgroups (otherwise a loop
in the graph introduces HNN extensions which are generally harder to deal
with regarding questions of linearity and residual finiteness, so this
would have been considered only on successful completion of the former
case).

This case looked promising because of the following result of Shalen, which
is \cite{sha} Proposition 1.3.
\begin{prop}
Let $G_1*_HG_2$ be a free product amalgamating the subgroup $H_1\leq G_1$
with $H_2\leq G_2$ via the isomorphism $\phi:H_1\rightarrow H_2$.
Suppose that $G_1$ and $G_2$ are both subgroups of $SL(n,\C)$ such that\\
(1) the matrices $\phi(h)$ and $h$ are the same for all $h\in H_1$,\\
(2) every $h\in H_1$ is a diagonal matrix, and\\
(3) for every $g_1\in G_1-H_1$ the bottom left hand entry is non zero, as
is the top right hand entry for all $g_2\in G_2-H_2$. 
Then $G_1*_HG_2$ can also be embedded in $SL(n,\C)$.
\end{prop}

The proof goes through exactly if we replace $SL(n,\C)$ with $GL(n,\C)$
throughout.
This allows amalgamation over abelian subgroups if this subgroup is
precisely the subgroup of diagonal elements. Hence this explains our
initial assumption in Section 3 that the peripheral elements are
diagonalisable. It was possible to show non trivial solutions to the
equations relating the diagonal entries of the peripheral elements
(whereupon we would look for free groups with a specified word
given by the non central peripheral element  
having these specified eigenvalues) except in the closed case when the final
piece was added. This will correspond to a leaf of the graph, thus the
final piece will be a surface with one boundary
component, thus forcing the extra conditions of determinant 1 which
mean that now the number of variables and equations are equal and we
have no guarantee of a non trivial solution.

In fact Yi Liu in \cite{liu} showed that a nonpositively curved
graph manifold $M$
also has fundamental group which is linear over $\Z$. As
this curvature condition holds if $M$ has non empty boundary, it is
only linearity of fundamental groups of closed graph manifolds which is
in question, so we do not comment further on the above approach.
Note that by \cite{leeb} Example 4.1 the graph manifolds
considered in this paper admit a metric of nonpositive curvature if and
only if we have $\sma{cc} i&j\\k&l\fma=\sma{rr} \pm 1&0\\0&\pm 1\fma$ or 
$\sma{rr} 0&\pm 1\\\pm 1&0\fma$ (where we allow for all possible cases of
signs) so the 3-manifolds in Theorem 6.2 do not possess metrics of non
positive curvature. In particular linearity of the fundamental group
(over any field) is unknown for all the graph manifolds considered in this
paper where $i,j,k,l$ are non zero. As for trying to increase the dimension
above 4 in Theorem 6.2, we can automate the process in Section 3 to
avoid ploughing through endless cases. We wrote a basic MAGMA program
which on being given the two partitions of basis vectors for each
eigenspace, outputted the determinant of the equations. For $i,j,k,l$ all
non zero integers this was always non zero for dimension 5. From this
it seems likely that there is no embedding in $GL(n,\C)$ for moderately
small $n$ where the peripheral elements are diagonalisable or have an
eigenvalue which is not a root of unity. However given the variety of
arguments that were employed in Section 6, we are much less sure whether
these groups embed in $GL(5,\C)$ or $GL(6,\C)$ if the peripheral
elements are allowed to have 
more complicated Jordan normal forms.

If one wants variations on the question of whether all 3-manifold groups
are linear, one can restrict the ring all the way from $\C$ to $\Z$ and
one can try and restrict the dimension needed. If we consider all closed
3-manifolds $M$ admitting a metric of non positive curvature then although
$\pi_1(M)$ is now known to be linear over $\Z$, it is not known if there
is a universal $n\in\N$ such that $\pi_1(M)\subseteq GL(n,\Z)$ or even
$GL(n,\C)$.

We finish with a few words on fibred 3-manifolds. The powerful results
mentioned above showing linearity over $\Z$ originate in Wise's
results on virtually special groups. As these groups will then be
virtually RFRS, virtual fibering of all the manifolds mentioned above
as having a fundamental group linear over $\Z$ is now established
because of Agol's result in \cite{agjt} (with the exception of
Siefert fibre spaces whose Siefert fibration has non zero Euler number, 
thus are always closed, which will not be virtually fibred). 
Recently Agol was also able to
use this and Wise's work to establish that Thurston's famous question
on whether $M$ has a finite cover that fibres over the circle holds for
hyperbolic 3-manifolds. Thus every compact
orientable irreducible 3-manifold with non empty boundary a union of tori
is virtually fibred but this is not quite true
for closed orientable irreducible 3-manifolds. 
There are examples of graph manifolds failing to have a metric
of non positive curvature which are fibred, such as a mapping torus of
a Dehn twist on a closed surface, or virtually fibred but not fibred
as are our examples in Theorem 6.2
by \cite{neuw} Theorem D and the example after Theorem E of the same paper.
Along with those Siefert fibre spaces just mentioned,
there are even closed graph manifolds failing to have a metric of non
positive curvature which are not
virtually fibred as shown in \cite{luwu}. 

Therefore we can
restrict the linearity question a little by considering only fibred
3-manifolds, where the fundamental group is $\pi_1(S_g)\rtimes_\alpha\Z$
and every such group is the fundamental group of a fibred 3-manifold.
We can ask whether all groups of the form $\pi_1(S_g)\rtimes_\alpha\Z$ are\\
(1) Linear over $\C$?\\
(2) Linear over $\Z$?\\
(3) Embeddable in $GL(4,\C)$?\\
Now this is known to hold for most such groups because of Thurston's
famous work showing that a pseudo Anosov homeomorphism of $S_g$ gives
rise to a hyperbolic structure on the corresponding mapping torus.
We finish by pointing out that the current situation is very different if
we replace $\pi_1(S_g)$ by a free group $F_n$. It is completely open
whether all groups of the form $F_n\rtimes_\alpha\Z$ are linear, even
for a fixed $n\geq 3$.

\end{document}